\documentclass[11pt]{amsart}
\usepackage{amsmath, mathrsfs,amsthm, amssymb,color,fullpage, mathtools,marvosym,enumerate,cancel}
\usepackage[utf8]{inputenc}
\usepackage{hyperref}
\usepackage{tikz}
\newcommand{\Z}{\mathbb{Z}}
\newcommand{\R}{\mathbb{R}}
 \def\beq{\begin{eqnarray}}
\def\eeq{\end{eqnarray}}

\newtheorem{theorem}{Theorem}
 
\newtheorem{lemma}{Lemma} 

\definecolor{pink}{rgb}{1,0,1}

\title{Three's company in six dimensions:  irreducible, isospectral, non-isometric flat tori}

\author{Gustav M\aa rdby}\address{Department of Mathematical Sciences \\Chalmers University of Technology and The University of Gothenburg\\SE-41296, Gothenburg} \email{mardby@chalmers.se}
 
\author{Julie Rowlett}
\address{Department of Mathematical Sciences \\Chalmers University of Technology and The University of Gothenburg\\SE-41296, Gothenburg}\email{julie.rowlett@chalmers.se}

\author{Felix Rydell}\address{Department of Mathematics, Royal Institute of Technology, 
Lindstedtsvägen 25, 114 28 Stockholm, Sweden}\email{felix.rydell@gmail.com}
 
\begin{document}

\begin{abstract}
In 1964, John Milnor, using a construction of two lattices by Witt, produced the first example of two flat tori that are not globally isometric and whose Laplacians for exterior forms have the same sequence of eigenvalues.  The aforementioned flat tori are sixteen-dimensional.  One is reducible while the second is irreducible.  In the ensuing years, pairs of non-isometric flat tori that share a common Laplace spectrum have been shown to exist in dimensions four and higher.  In dimensions three and lower, Alexander Schiemann proved in 1994 that any flat tori that are isospectral are in fact isometric, so four is the lowest dimension in which such pairs exist. Using a four-dimensional such pair, one can easily construct an eight-dimensional such triplet. However, triplets of mutually non-isometric 
flat tori that share a common Laplace spectrum in dimensions 4, 5, 6, and 7 have eluded researchers - until now.  We present here the first example. 
\end{abstract}

\maketitle

How does geometry influence physics?  In other words, to what extent does an object's shape influence its physical properties?  This question is not only the focus of a vast amount of scientific research but is also important to numerous tangible situations, such as green building, aerodynamic design, medical imaging, and drug delivery.  A quintessential mathematical object in all of these contexts as well as further physical processes including quantum phenomena,  heat conduction, vibrations, sound, electromagnetism, and diffusion is the Laplace operator. The spectrum of this operator plays an essential role in all of these contexts and physical processes.  This motivates our quest to understand the intricate interplay between the Laplace spectrum - and therewith numerous physical processes - and geometry.

The simplest example of how geometry influences the Laplace spectrum is conveyed by music.  The beautiful (or terrible, depending on the player) sound of a piano is created by the vibration of its strings.  The ends of each string are held fixed, whereas the rest of the string vibrates, producing its sound. Given an initial position and velocity of a specific string, its position at later times is obtained by solving the wave equation.  This can then be used to deduce information about the sounds produced by the vibrating string.  One method for solving the wave equation is to first obtain all standing waves and then superimpose these using the Fourier coefficients obtained from the initial data (the initial position and velocity).  We obtain these standing waves mathematically by solving the one-dimensional Helmholtz equation and therewith calculating the eigenvalues of the Laplace operator for our string.  The collection of these eigenvalues is the Laplace spectrum of our string.  When we calculate this Laplace spectrum, we obtain the length of the string. 
If we identify a string with a circle of the same length, we can therefore prove that if two circles have the same Laplace spectrum, then they have the same length. Equivalently, if two one-dimensional flat tori share a common Laplace spectrum, then they are identical.  

Of course a real string is three-dimensional, and indeed the cross-section of the string also has important implications for the sounds it produces.  While it is not possible to analytically compute the Laplace spectrum of an arbitrarily shaped, realistic three-dimensional string, it is possible to analytically compute the Laplace spectrum of a \em flat torus \em of any dimension.    A flat torus is a compact Riemannian manifold obtained as a quotient of flat Euclidean space by a full-rank lattice.  Let $\R^n$ denote the Euclidean space of dimension $n$, and let $\Z^n \subset \R^n$ be the set of elements of $\R^n$ whose components are all integers.  Let $M$ be an invertible $n \times n$ matrix with real-valued entries.  A \em full-rank lattice \em is a set $M \Z^n :=\{Mz: z\in \Z^n\}.$  The matrix $M$ is a \em basis matrix.  \em  It is not unique, but all basis matrices are related in the sense that every other basis matrix can be expressed as $MB$ for some unimodular matrix $B$. 
For a full-rank lattice $\Gamma \subset \R^n$, there is an associated full-rank lattice known as the dual, often denoted $\Gamma^*$.  This consists of all elements of $\R^n$ whose scalar product with any arbitrary element of $\Gamma$ is an integer.  It is then a classical exercise to prove that the Laplace spectrum of the flat torus is equal to the squares of the lengths of the elements of the dual lattice multiplied by a constant factor $4 \pi^2$.  By Poisson's summation theorem, to investigate the Laplace spectrum of a flat torus, it is equivalent to study the lengths of the vectors in full-rank lattices.  The set of lengths of vectors in a full-rank lattice, counted with multiplicity, is known as its \em length spectrum.  \em  In this way, the study of the Laplace spectrum, initially motivated by physical considerations, has lead to the study of a geometric quantity, the \em length spectrum \em of a full-rank lattice.  We may simply refer to this as the \em spectrum \em of the lattice, keeping in mind the equivalence between the Laplace spectrum of a flat torus and the length spectrum of the associated lattice.

Lattices are not only instrumental in the study of the Laplace spectrum, but also beautiful objects with utility in a myriad of contexts.  For example, they describe structures in biology, chemistry, and physics, are used in communication systems, can provide robust algorithms for post quantum encryption, and are essential objects in analytic number theory.  A closely related object that is also useful for problems in number theory is a quadratic form.  Quadratic forms are essential to many fields - not only number theory - but also differential geometry, data science, machine learning, optimization, and physics. Given an $n$-dimensional full-rank lattice $\Gamma$ with basis matrix $A$, let $Q = A^T A$.  Then, we identify $Q$ with the quadratic form that acts on $x \in \R^n$ via $Q(x) = x^T Q x$.  If $P$ is an $n \times n$ matrix with $B^T Q B = P$, for a unimodular matrix $B$, then we say that the quadratic form defined via $P$ in the analogous way is \em integrally equivalent \em to $Q$.  Since any other basis matrix for $\Gamma$ can be expressed as $AB$ for a unimodular matrix $B$, we therefore associate the class of integrally equivalent quadratic forms with the lattice $\Gamma$.  For $t \in \R$, the $t$-th \textit{representation number}, often denoted $R(Q, t)$, is the number of distinct $x \in \Z^n$ such that $Q(x)=t$.  There is a natural bijection between the length spectrum of the flat torus $\R^n / \Gamma$ and the representation numbers of this equivalence class.  For $x\in \mathbb{Z}^n$, the length of the lattice vector $\|Ax\|$ is mapped to $Q(x)=x^TA^TAx=\|Ax\|^2$.  In this way one can also see that the representation numbers of integrally equivalent quadratic forms are identical.  Consequently, two flat tori are isospectral if and only if the representation numbers of their associated quadratic forms are identical.

It is then natural to ask, how does the `shape' of a lattice, and therewith the shape of the associated flat torus, influence its Laplace spectrum?  Equivalently, to what extent can we identify a quadratic form if we know its representation numbers?  Perhaps the most fundamental question in this direction is:  if the Laplace spectra of two flat tori are identical, then do the flat tori have the same shape?  Mathematically, we would say that the flat tori have the same shape if they are isometric as Riemannian manifolds.  This is equivalent to the existence of an orthogonal transformation that takes the first lattice to the second lattice.  As mentioned above, it is a classical exercise to prove that if two one-dimensional flat tori have identical Laplace spectra, then they are isometric.  It is slightly more difficult to prove that this also holds in two dimensions, but is still a feasible task for anyone with a background in mathematics. The first example of two flat tori which are isospectral but \em not \em isometric is a 16-dimensional pair found by Milnor in 1964 \cite{milnor1964eigenvalues}. 

After Milnor's example was found, a question arose: what is the smallest dimension in which isospectral non-isometric flat tori exist?  In 1967 Kneser found a 12-dimensional pair \cite{kneser1967lineare}.  It took ten years for Kitaoka to find a pair in dimension 8 \cite{kitaoka1977positive} in 1977. Nine years later Conway and Sloane managed to find six- and five-dimensional examples \cite{conway1992four}. In 1990, Schiemann used a computer to find a four-dimensional example \cite{schiemann1990beispiel}. One year later, Shiota found another four-dimensional pair \cite{shiota1991theta}, and in the same year Earnest and Nipp found yet another pair \cite{earnest1991theta}. In 1992, Conway and Sloane found an infinite family of four-dimensional pairs \cite{conway1992four} \cite{cervino2011conway}. Then, in 1994 Schiemann showed, using an advanced computer algorithm now known as \textit{Schiemann's algorithm} \cite{nilsson2023isospectral}, that there are no pairs in three dimensions \cite{schiemann1994ternare, schiemann1997ternary}. Thus, the smallest dimension in which isospectral non-isometric flat tori exist is four. 

Now, we may wonder, how \em many \em flat tori can share a common Laplace spectrum?  In 1978, Wolpert showed that any collection of mutually isospectral non-isometric flat tori is finite \cite{wolpert1978eigenvalue}. Suwa-Bier improved this result in 1984 by showing that the supremum over the sizes of all such collections in any given dimension is also finite \cite{suwa1984positiv}.  However, it is not clear how to extract an explicit upper bound bound from  \cite{wolpert1978eigenvalue, suwa1984positiv}, and we are also unaware of a conjecture concerning the precise value of this supremum as a function of the dimension. As soon as one has pairs of isospectral non-isometric flat tori in a given dimension, one can begin constructing increasingly larger families of isospectral yet mutually non-isometric flat tori. For instance, if $(\Gamma,\Lambda)$ is Schiemann's four-dimensional pair of incongruent lattices with equal spectra, then 
\begin{align}
    (\Gamma\oplus \Gamma, \Gamma\oplus \Lambda, \Lambda\oplus \Lambda)
\end{align}
is an eight-dimensional triplet of incongruent lattices with equal spectra. Here, $\oplus$ denotes the direct sum of the canonical embeddings of the four-dimension lattices in the first four, respectively last four, coordinate directions of $\R^8$.   
To the best of our knowledge, no such triplets (or quadruplets) have been demonstrated in dimensions lower than eight.

\section*{Linear codes}
The main idea for this paper is to search for triplets among linear codes.  Let $q$ and $n$ be positive integers.  A \textit{linear code} is a subgroup of the additive group $(\Z/q\Z)^n$, and its elements are called \textit{codewords}. Every linear code of $(\Z/q\Z)^n$ has a so-called \textit{generator matrix}, which is a (not unique) $n \times n$ matrix whose rows consist of generators of the subgroup. Since any subgroup can be generated by $n$ or fewer elements of the group, in case the subgroup has less than $n$ generators, a generator matrix will have some rows containing only zeros. These rows are simply copies of the identity element of the group in order to create a square matrix.  For example, the subgroup consisting of only the identity element of the group has the generator matrix that is the $n \times n$ matrix with all entries equal to zero.  

To understand the relationship between linear codes and integer lattices, we define the projection map $\pi_q : \Z^n \to (\Z/q\Z)^n$ by 
 $   \pi_q(z) = z \mod q.$
Here, mod $q$ acts coordinate-wise and refers to the integers modulo $q$. For a lattice $L\subset \Z^n$, $\pi_q(L)$ is a linear code, and for a linear code $C$, $\pi_q^{-1}(C)$ is a lattice. Further, the following lemma shows that every lattice is the pre-image of at least one linear code. 

\begin{lemma}[\cite{nilsson2023isospectral}, Section 2.1 \& 3.3] \label{lemma:code_to_lattice} Let $L$ be a full-rank lattice in $\R^n$. Then
 $   L = \pi_q^{-1}(\pi_q(L)) $
if and only if $q\Z^n\subset L$.  If $L=M \Z^n$, then $\det(M) \Z^n \subset L$.
\end{lemma}

For each $q$ and $n$, the set of linear codes is finite. We can therefore conduct a systematic search through integer lattices by generating all linear codes for $q=1,2,3,\ldots$.  We have developed and implemented an algorithm for this purpose.  As observed by Conway \cite[p. 40--42]{conway1997sensual}, one can deduce isospectrality of lattices by looking at weight distributions of codes.  We say that two linear codes have equal \em weight distributions \em if there is a bijection between codewords that preserves the codewords up to permutations and signs (modulo $q$).  For $q=2$, this is equivalent to the bijection preserving the number of non-zero elements.  It is straightforward to check that if $C_1$ and $C_2$ are linear codes in $(\Z/q\Z)^n$ with equal weight distributions, then $\pi_q ^{-1} (C_i)$, $i=1,2$ are isospectral lattices.

\section*{A six-dimensional triplet}
For $q = 5$ and $n = 6$ we found, using our algorithm, three linear codes $C_1, C_2, C_3$ with equal weight distributions. Their corresponding lattices are
\begin{equation} \label{eq:lattice_triple}
\begin{split} 
    L_1 &= \pi_{5}^{-1}(C_1) = \begin{bmatrix}1&0&0&0&0&0\\ 0&1&0&0&0&0\\ 0&0&1&0&0&0\\ 1&1&0&5&0&0\\2&0&1&0&5&0\\1&2&1&0&0&5\end{bmatrix}\Z^6, \\
    L_2 &= \pi_{5}^{-1}(C_2) = \begin{bmatrix}1&0&0&0&0&0\\ 0&1&0&0&0&0\\ 0&0&1&0&0&0\\ 2&1&0&5&0&0\\0&1&1&0&5&0\\3&2&1&0&0&5\end{bmatrix}\Z^6, \\
    L_3 &= \pi_{5}^{-1}(C_3) = \begin{bmatrix}1&0&0&0&0&0\\ 0&1&0&0&0&0\\ 0&0&1&0&0&0\\ 2&1&0&5&0&0\\0&1&1&0&5&0\\2&3&1&0&0&5\end{bmatrix}\Z^6.
\end{split}
\end{equation}
In the following sections, we will formally prove that these correspond to three mutually isospectral and non-isometric irreducible flat tori.

\section*{Isospectrality}
To show that the three lattices $L_i = A_i\Z^6$ in \eqref{eq:lattice_triple} correspond to three isospectral flat tori, we instead show the equivalent fact that the quadratic forms
\begin{equation} \label{eq:quad_forms}
\begin{split}
    Q_1 &= A_1^TA_1 = \begin{bmatrix}7&3&3&5&10&5\\ 3&6&2&5&0&10\\ 3&2&3&0&5&5\\ 5&5&0&25&0&0\\10&0&5&0&25&0\\5&10&5&0&0&25\end{bmatrix}, \\ 
    Q_2 &= A_2^TA_2 = \begin{bmatrix}14&8&3&10&0&15\\ 8&7&3&5&5&10\\ 3&3&3&0&5&5\\ 10&5&0&25&0&0\\0&5&5&0&25&0\\15&10&5&0&0&25\end{bmatrix}, \\ 
    Q_3 &= A_3^TA_3 = \begin{bmatrix}9&8&2&10&0&10\\ 8&12&4&5&5&15\\ 2&4&3&0&5&5\\ 10&5&0&25&0&0\\0&5&5&0&25&0\\10&15&5&0&0&25\end{bmatrix}
\end{split}
\end{equation}
have the same representation numbers.  This can be done using the following result, which follows from Hecke's identity theorem for modular forms.  Recall that a quadratic form $Q$ is called \textit{even} if every element in $Q$ is an integer and the diagonal elements are even. If $Q$ is even and positive definite, we define $N_Q$ to be smallest positive integer such that $N_QQ^{-1}$ is even.

\begin{theorem}[{\cite[Thm. 3.6]{nilsson2023isospectral}}] \label{lemma:isospectral_quadratic_forms}
Let $P$ and $Q$ be two even positive definite quadratic forms in $2k$ variables. Assume $\det(P) = \det(Q)$, $N_P = N_Q$, and that the $t$-th representation numbers of $P$ and $Q$ coincide for $0\le t\le \mu_0(N_P)k/6 + 2$, where  
\begin{equation}
        \mu_0(N) = N \prod_{p|N, \text{ prime }} \left(1 + \frac{1}{p}\right). \label{eq:mu0}
    \end{equation}
   Then all representation numbers for $P$ and $Q$ are the same. 
\end{theorem}

We point out that a subtly misstated version of this result appears in \cite[Cor. 3.7]{nilsson2023isospectral}.  The salient point is that  to show that two even quadratic forms in $2k$ variables are isospectral, it is not enough to check that their representation numbers up to $\frac{\mu_0(N_P)k}{12} + 1$ are the same. Instead, one must check that the first $\frac{\mu_0(N_P)k}{12} + 1$ \textit{even} representation numbers, i.e. the representation numbers up to $2(\frac{\mu_0(N_P)k}{12} + 1)$, are the same.  The result can easily be extended odd dimensions, as explained in \cite[Section 3.2]{nilsson2023isospectral}.

\begin{theorem}
    The three quadratic forms $Q_i$ given by \eqref{eq:quad_forms} have the same representation numbers, hence the corresponding flat tori are isospectral.
\end{theorem}
\begin{proof}
    While $Q_i$ are not even, $2Q_i$ are, and $Q_i$ have the same representation numbers if and only if $2Q_i$ do. Now, the reader may verify that $\det(2Q_i) = 10^6$ and $N_{2Q_i} = 100$, $i = 1, 2, 3$. Moreover, $\mu_0(100) = 180$ and $\mu_0(N_{2Q_i})3/6 + 2 = 92$, so it remains to show that $2Q_i$ have the same representation numbers up to 92. This follows from Tables \ref{Table1} to \ref{Table7}, which were obtained from a well-tested computer program \href{https://github.com/julierowlett/Isospectrality-and-Isometry-Check}{that can be accessed in this Github repository.}
\end{proof}

\begin{table}[h!]
\begin{center}
\caption{The representation numbers of $2Q_i$ between 0 and 16.}
\label{Table1}
\begin{tabular}{ |c|c|c|c|c|c|c|c|c|c| } 
 \hline
  $k$ & 0 & 2 & 4 & 6 & 8 & 10 & 12 & 14 & 16  \\ 
 \hline
 $R(2Q_i,k)$ & 1 & 0 & 0 & 2 & 2 & 2 & 2 & 10 & 8  \\ 
 \hline
\end{tabular}
\end{center}
\end{table}

\begin{table}[h!]
\begin{center}
\caption{The representation numbers of $2Q_i$ between 18 and 32.}
\label{Table2}
\begin{tabular}{ |c|c|c|c|c|c|c|c|c| } 
 \hline
  $k$ & 18 & 20 & 22 & 24 & 26 & 28 & 30 & 32  \\ 
 \hline
 $R(2Q_i,k)$ & 4 & 12 & 16 & 22 & 18 & 20 & 32 & 30  \\ 
 \hline
\end{tabular}
\end{center}
\end{table}

\begin{table}[h!]
\begin{center}
\caption{The representation numbers of $2Q_i$ between 34 and 46.}
\label{Table3}
\begin{tabular}{ |c|c|c|c|c|c|c|c| } 
 \hline
  $k$ & 34 & 36 & 38 & 40 & 42 & 44 & 46 \\ 
 \hline
 $R(2Q_i,k)$ & 34 & 46 & 52 & 48 & 28 & 78 & 102 \\ 
 \hline
\end{tabular}
\end{center}
\end{table}

\begin{table}[h!]
\begin{center}
\caption{The representation numbers of $2Q_i$ between 48 and 58.}
\label{Table4}
\begin{tabular}{ |c|c|c|c|c|c|c| } 
 \hline
  $k$ & 48 & 50 & 52 & 54 & 56 & 58 \\ 
 \hline
 $R(2Q_i,k)$ & 54 & 70 & 68 & 120 & 124 & 64 \\ 
 \hline
\end{tabular}
\end{center}
\end{table}

\begin{table}[h!]
\begin{center}
\caption{The representation numbers of $2Q_i$ between 60 and 70.}
\label{Table5}
\begin{tabular}{ |c|c|c|c|c|c|c| } 
 \hline
  $k$ & 60 & 62 & 64 & 66 & 68 & 70 \\ 
 \hline
 $R(2Q_i,k)$ & 104 & 124 & 160 & 112 & 110 & 184 \\ 
 \hline
\end{tabular}
\end{center}
\end{table}

\begin{table}[h!]
\begin{center}
\caption{The representation numbers of $2Q_i$ between 72 and 82.}
\label{Table6}
\begin{tabular}{ |c|c|c|c|c|c|c|c|c| } 
 \hline
  $k$ & 72 & 74 & 76 & 78 & 80 & 82 \\ 
 \hline
 $R(2Q_i,k)$ & 108 & 162 & 230 & 164 & 200 & 132 \\ 
 \hline
\end{tabular}
\end{center}
\end{table}

\begin{table}[h!]
\begin{center}
\caption{The representation numbers of $2Q_i$ between 84 and 92.}
\label{Table7}
\begin{tabular}{ |c|c|c|c|c|c| } 
 \hline
  $k$ & 84 & 86 & 88 & 90 & 92 \\ 
 \hline
 $R(2Q_i,k)$ & 220 & 366 & 202 & 170 & 236 \\ 
 \hline
\end{tabular}
\end{center}
\end{table}

\section*{Non-isometry}
Let $\Gamma_1 = A_1\Z^n$ and $\Gamma_2 = A_2\Z^n$ be two full-rank lattices in $\R^n$. The flat tori $\R^n/\Gamma_1$ and $\R^n/\Gamma_2$ are isometric as Riemannian manifolds if and only if the lattices $\Gamma_1, \Gamma_2$ are \textit{congruent}, meaning that $C\Gamma_1 = \Gamma_2$ for some orthogonal matrix $C \in O_n(\R)$. This holds if and only if the quadratic forms $Q_1 = A_1^TA_1$ and $Q_2 = A_2^TA_2$ are integrally equivalent.
 
A convenient way to check whether two quadratic forms are integrally equivalent uses the following result, which is a generalization of \cite[Cor. 3.3]{nilsson2023isospectral}.
\begin{lemma} \label{lemma:IEQF}
    Let $Q_1, Q_2$ be two positive definite $n$-dimensional quadratic forms. Let $\lambda_{\min}$ be the smallest eigenvalue of $Q_1$. If $B^TQ_1B = Q_2$ for some unimodular matrix $B$ with columns $b_j$, then 
    \begin{equation}
        b_i^TQ_1b_j = (Q_2)_{ij}, \,\, i,j = 1,\dots,n.
    \end{equation}
    Moreover, $\|b_j\|^2 \leq (Q_2)_{jj}/\lambda_{\min}$ for each $j = 1,\dots,n$.
\end{lemma}
Since the elements in the unimodular matrix $B$ are integers, there are only finitely many such matrices satisfying the conditions in Lemma \ref{lemma:IEQF}. 

\begin{theorem}
    The three quadratic forms given by \eqref{eq:quad_forms} are not integrally equivalent, hence the corresponding flat tori are non-isometric.
\end{theorem}
\begin{proof}
    The smallest eigenvalue of $Q_1$ is $\frac{263}{400}$. Thus, if $B^TQ_1B = Q_2$ for some unimodular matrix $B$ with columns $b_j$, then by Lemma \ref{lemma:IEQF} we have
    \begin{align}  \begin{aligned}  
        \|b_1\|^2 &\leq \frac{5600}{263}, \, &\|b_2\|^2 &\leq \frac{2800}{263}, \,& \|b_3\|^2 & \leq \frac{1200}{263}, \\
        \|b_4\|^2 &\leq \frac{10000}{263}, \, &\|b_5\|^2 &\leq \frac{10000}{263}, \,&\|b_6\|^2 &\leq \frac{10000}{263}.
    \end{aligned}
    \end{align}
    However, using a well-tested computer program \href{https://github.com/julierowlett/Isospectrality-and-Isometry-Check}{located in this Github repository}, we find that no such matrix $B$ satisfies $b_i^TQ_1b_j = (Q_2)_{ij}$ for all $i,j$. Therefore $Q_1$ and $Q_2$ are not integrally equivalent. The result follows similarly for $Q_1, Q_3$ and $Q_2, Q_3$. 
\end{proof}

\section*{Irreducibility}
A lattice $\Gamma$ is called \textit{reducible} if it is the direct sum of two lower-dimensional lattices $\Gamma_1$ and $\Gamma_2$. We write this as $\Gamma = \Gamma_1 \oplus \Gamma_2$. If a lattice is not reducible, we say it is \textit{irreducible}. We would like to know if our 6-dimensional triplet given by \eqref{eq:lattice_triple} consists of reducible or irreducible lattices.  A motivation for this question is the fact that many of the examples of isospectral non-isometric flat tori have been produced using one reducible lattice and one irreducible lattice, including Milnor's 16-dimensional pair \cite{milnor1964eigenvalues} and Conway's 6-dimensional pair \cite{conway1997sensual}.

\begin{theorem}
    The three lattices $L_i$ given by \eqref{eq:lattice_triple} are irreducible.
\end{theorem}
\begin{proof}
    We show that $L_1$ is irreducible. The proof for $L_2$ and $L_3$ is similar. Suppose $L_1$ is reducible, so that $L_1 = U \oplus V$ for some sublattices $U$ and $V$ of dimension less than six. Then every vector in $U$ is orthogonal to every vector in $V$. 
    
    Now, using a computer, we find that the shortest non-zero vectors in $L_1$ are $v_1^{\pm} = \pm(0,0,1,0,1,1)$. For simplicity, we write $v_1$ to indicate $v_1 ^+$, and use the analogous notation for $v_i$ below for $i=2,3,4,5,6$ as well as $w_4$.
    The vectors $v_1^{\pm}$ are in either $U$ or $V$ since a sum of a non-zero vector in $U$ and $V$ has length strictly greater than $\|v_1\|$. Assuming $v_1 \in V$, we then note that the shortest non-zero vectors in $L_1$ apart from $v_1^{\pm}$ are $v_2^{\pm} = \pm (1,0,-1,1,1,0)$. While $v_1$ and $v_2$ are orthogonal, they do in fact belong to the same sublattice $V$. To see this, note that the shortest non-zero vectors in $L_1$ apart from $v_1^{\pm}$ and $v_2^{\pm}$ are $v_3^{\pm}= \pm (0,1,-1,1,-1,1)$. If $v_1 \in V$ and $v_2 \in U$, then $v_3$ is neither in $U$ nor in $V$ since $v_3$ is not orthogonal to $v_1$ or $v_2$. So, $v_3 = u' + v'$ for some $u' \in U$ and $v' \in V$. Then
    \begin{equation}
        5 = \|v_3\|^2 = \|u'\|^2 + \|v'\|^2 \geq \|v_2\|^2 + \|v_1\|^2 = 7,
    \end{equation}
    which is a contradiction. It follows that both $v_1$ and $v_2$ are in $V$. Then $v_3$ is not in $U$, because it is not orthogonal to $v_1$ and $v_2$. If it were a sum of non-zero elements in $U$ and $V$, its length would be strictly greater than the shortest vector in $U$. This is a contradiction, because this shortest vector is at least of length $\|v_3\|$. Therefore $v_3\in V$.

    Next, the shortest vectors in $L_1$ which are linearly independent of $v_1, v_2, v_3$ are $v_4^{\pm} = \pm (2,-1,0,1,-1,0)$ and $w_4^{\pm}= \pm (1,-1,1,0,-2,0)$. Since $v_4, w_4$ are not orthogonal to $v_1$, it follows by the same argument as before that $v_4, w_4 \in V$. Similarly, the shortest vectors in $L_1$ linearly independent of $v_1, v_2, v_3, v_4$ are $v_5^{\pm} = \pm (1,0,1,2,1,-1)$, which again are in $V$ because they are not orthogonal to $v_1$. Finally, the shortest vectors in $L_1$ which are linearly independent of $v_1,v_2,v_3,v_4,v_5$ are $v_6^{\pm}= \pm(2,1,1,-2,0,0)$, which are also in $V$ since they are not orthogonal to $v_1$. Therefore, we have six linearly independent vectors in the same sublattice, which contradicts that the sublattice has dimension less than six. We conclude that no such sublattices of $L_1$ exist.  
\end{proof}

\section*{Outlook}
For any $n \geq 1$, the $n$-th \textit{choir number} $\flat_n$ is defined as the maximum number $k$ such that there exist $k$ mutually isospectral and non-isometric $n$-dimensional flat tori. The name \em choir number \em is derived from the connection of Laplace spectra to music as detailed in the introduction. In 1990-1994, Schiemann proved that $\flat_1 = \flat_2 = \flat_3 = 1$ and $\flat_n \geq 2$ for $n \geq 4$ \cite{schiemann1990beispiel,schiemann1994ternare}.  Now, let $(\Gamma_1,\Gamma_2)$ be a pair of four-dimensional incongruent and isospectral lattices.  Then
\begin{align}
1\cdot \Gamma_{i_1}\oplus 2\cdot \Gamma_{i_2} \oplus \cdots \oplus n \cdot \Gamma_{i_n},
\end{align}
where $\lambda \cdot \Gamma$ denotes the scaling of the lattice $\Gamma$ by $\lambda$, are $4n$-dimensional, incongruent, and isospectral for any $i_j\in \{1,2\}$. The scalings are included to make sure that the lattices are incongruent, and they may be chosen arbitrarily as long as they are non-zero and distinct up to sign. Since they are $2^n$ in number, we conclude that $\flat_{4n} \geq 2^n$. Using our six-dimensional triplet, we similarly find that $\beta_{6n}\ge 3^n$. As $3^{1/6}>2^{1/4}$, this yields a tighter lower bound.  A similar argument shows that the choir numbers are supermultiplicative, in the sense that $\flat_{m+n} \geq \flat_m \flat_n$ for all $m, n \geq 1$, including $m=n$. Thus, the choir numbers grow faster than any polynomial.  On the other hand, the choir numbers are known to be finite \cite{wolpert1978eigenvalue, suwa1984positiv}, although we are not aware of an explicit bound for their values in terms of the dimension $n$.  The existence of a six-dimensional triplet implies that $\flat_6 \geq 3$.  We now conjecture that $\flat_4 = \flat_5 = 2$ and $\flat_6 = 3$.  We are currently working on a proof of this and investigating the values of the choir numbers in higher dimensions.


\end{document}